\documentclass{amsart}
\usepackage{graphicx} % Required for inserting images
\usepackage{enumerate,amsthm,amsmath,amsfonts,amssymb,mathrsfs,hyperref,dutchcal,xcolor}
\usepackage[cal=cm]{mathalfa}
\usepackage{fullpage}
\newtheorem{thm}{Theorem}
\newtheorem{defi}{Definition}

\newtheorem{lem}{Lemma}

\newtheorem{rmk}{Remark}

\title[Gevery KAM equilibria]{Gevrey KAM equilibria for quasi-periodic  long-range Frenkel-Kontorova models}

\author{Yujia An}
\address{School of Mathematical Sciences, Beijing Normal University,
No. 19, XinJieKouWai St., HaiDian District, Beijing 100875, P. R. China}
\email{yjan@mail.bnu.edu.cn}

\author{Xifeng Su}
\address{School of Mathematical Sciences, Laboratory of Mathematics and Complex Systems (Ministry of Education)\\
Beijing Normal University,
No. 19, XinJieKouWai St., HaiDian District, Beijing 100875, P. R. China}
\email{xfsu@bnu.edu.cn, billy3492@gmail.com}

\date{\today}

\begin{document}
\maketitle

\begin{abstract}

We consider models of one-dimensional chains of non-nearest neighbor and many-body  interacting particles subjected to quasi-periodic media. 
We extend the results in \cite{12Su&delaLlavelongrange} from analytic to Gevrey regularity potentials. 
More precisely, we establish an a posteriori KAM theorem showing that in the Gevrey topology, given an approximate solution of equilibrium equation, 
which satisfies some appropriate non-degeneracy conditions and decay property, then there is a true solution nearby and the solution preserves both the quasi-periodicity and Gevrey regularity. 
The method of proof is based on a combination of  quasi-Newton methods and delicate estimates in spaces of Gevrey functions.

\end{abstract}

\tableofcontents

\section{Introduction}
The Frenkel-Kontorova model, originating from solid state physics, describes a chain of particles interacting with their nearest neighbors. The long distance generalization has a  clear physical meaning since the deformation produced by a dislocation causes effects at long distance (see e.g. \cite{FK39, Nabarro52}).  

More precisely, the configuration of the system is given by a sequence $\{u_n\}_{n\in\mathbb Z}$, where $u_n\in\mathbb R$ denotes the position of the $n-$th particle. The formal energy of the configuration $\{u_n\}_{n\in\mathbb Z}$ is given by 
\begin{equation}\label{long-rangemodel}
\mathscr S(\{u_n\}_{n\in\mathbb Z})=\sum_{i\in\mathbb Z}\sum_{L=0}^{\infty}\widehat H_L(u_i\alpha,u_{i+1}\alpha,\cdots,u_{i+L}\alpha),
\end{equation}
where $\alpha\in \mathbb{R}^d$ and $\widehat H_L: (\mathbb{T}^d)^{L+1} \rightarrow \mathbb{R}$. We remark that the form of $\widehat H_L$ encodes both the quasi-periodic properties of the media and the long-range interactions.

The physical states are then selected to be the critical points of the above formal energy functional, which are called the equilibrium configurations. 
We are interested in finding ``line-like" equilibrium configurations of the form
\begin{equation}\label{hullfunction}
    u_n=h(n\omega),\,\, \forall~ n\in\mathbb Z,
\end{equation}
 where $\omega\in\mathbb R$ is called the rotation number, and $h(\theta)=\theta+\tilde h(\theta)$ with $\tilde h$ is a quasi-periodic or an almost-periodic function which corresponds to quasi-periodic or almost-periodic media.
  The function $h$ is often referred to as the ``hull function" of the configuration in solid state physics. 
  
  In \cite{08delaLlave}, \cite{12Su&delaLlavelongrange}, \cite{24An&delaLlave&S&W&Y}, the corresponding KAM-type theorems (the existence of equilibrium configurations) for long-range and many-body interactions were established  respectively for cases where the interactions $\widehat H_L$ are
  \begin{itemize}
  \item analytic and periodic;
  \item analytic/Sobolev and quasi-periodic;
  \item analytic and almost-periodic.
  \end{itemize}
 
 Note that the long-range interactions lead to interesting physical effects (due to electrical, magnetic
or exchange forces, etc.) and the nearest-neighbor interactions could be thought of as an approximation. But, in general,
 they also make the corresponding models \eqref{long-rangemodel} impossible to obtain a dynamical system interpretation of the equilibrium
equations. Even when they do, the Hamiltonian description may be very singular or the number of freedom may be very large, so that the classical Hamiltonian KAM theory (such as transformation theory) does not apply.
  
On the other side, the quasi-periodic (or almost-periodic) nature of the interactions also makes it impossible to give
a straightforward dynamical meaning to the equilibrium equations (see \cite{12Su&delaLlave, LSZ16}). 
Moreover,  the quasi-periodic solutions correspond to the corrector considered in homogenization theory and  the paper \cite{LS03} pointed out that the addition of an extra incommensurate frequency causes a lack of compactness and breakdown of the arguments in the periodic case. This indicates that the variational construction of solutions may fail to exist.

Recently,  Gevrey functions has received a lot of interest  since those functions are related to many deep theorems of dynamical systems (e.g. KAM theorem, Nekhoroshev estimates). 
In particular, even if one begins with an analytic (or even polynomial) problem, several objects of interest such as asymptotic expansions  are only Gevrey \cite{BC19, BD22}.
  
  A natural question arises: 
  for interactions with Gevrey regularity, does the Gevrey KAM equilibria exist?

  When dealing with Gevrey regularity problems, two primary approaches are commonly employed. The first approach relies on approximation by analytic functions. 
  As established in \cite{03Popov}, this method involves approximating Gevrey functions by analytic ones and then applying classical KAM techniques.
  The second approach is based on direct norm estimates. In \cite{03M.J-Pierre&S.David}, Marco and Sauzin introduce an alternative Gevrey norm to analyze quasi-convex near-integrable Hamiltonian systems. Subsequent works \cite{17Lopes&Joao, 24You&Yuan} demonstrate that this norm admits an ``equivalence" in terms of  Fourier series, enabling techniques analogous to those used in analytic settings.

  In this work, adopting the direct norm approach, we establish the existence of Gevrey smooth quasi-periodic equilibrium configurations for the system with Gevrey class quasi-periodic interactions. Indeed, we prove that when the potential $\widehat H_L$ belongs to the Gevrey smoothness class and is quasi-periodic in the spatial variable, the corresponding equilibrium solutions inherit both the quasi-periodicity and Gevrey regularity of the interactions.

Moreover, once the existence of such equilibrium solutions is established, we immediately obtain a continuous family of equilibria. Specifically, for any solution of the form $u_n=n\omega+\tilde h(n\omega)$, where $\tilde h(\theta)=\hat h(\alpha\theta)$ with $\alpha\in\mathbb R^d$, the configuration $u_n^{\phi}=n\omega+\hat h(n\omega\alpha+\phi)$ remains an equilibrium for arbitrary $\phi\in\mathbb R$ .

Physically, this continuum of solutions indicates that the system possesses a sliding degree of freedom - the entire chain can undergo phase shifts ($\phi$-variations) while maintaining equilibrium, requiring vanishingly small energy input. Furthermore, when a continuous family of equilibria exists in our framework, the results of \cite{17delaLlave&Su} demonstrate that the equilibria we get are also the minimizers in the sense that any localized (compactly supported) perturbation of these equilibria increases the total energy of the system.

Before ending this section, we would remark that the way that we have formulated the result is an a posteriori theorem with which
given an approximate solution (e.g., obtained numerically) we can prove that there is
a true solution nearby. The method of proof is an iterative method which converges
quadratically in the Gevrey topology.  As in \cite{BlassdelaLlave13}, we point out that the iterative step described here can be implemented as an
efficient algorithm (low storage requirements and low operation count).

\section{Preliminary}
\subsection{Function spaces}
We consider the $d-$dimensional torus $\mathbb T^d:=\mathbb R^d/(2\pi\mathbb Z)^d$ for $d\ge2$, which serves as our fundamental space for studying quasi-periodic phenomena. We will investigate the properties and applications of Gevrey-class functions.

\begin{defi}
    Given $k\in\mathbb Z^d$, $\beta\ge1$, define %$|k|_1=|k_1|+\cdots+|k_d|$, 
    $|k|_{\beta}=|k_1|^{\frac1\beta}+\cdots+|k_d|^{\frac1\beta}$.
\end{defi}
   Obviously, for all $\beta\ge1$, we have $|k|_1\le|k|_{\beta}^{\beta}$.

\begin{defi}
    For $\beta\ge1$, a function $f\in C^{\infty}(\mathbb T^d)$ is called a Gevrey-$\beta$ function if there exists a constant $B>0$ such that $\sup_{k\in\mathbb N^d}\frac{B^{|k|}}{k!^{\beta}}\|\partial^k f\|_{C^0(\mathbb T^d)}<+\infty$. Denote by $G_{\beta}(\mathbb T^d)$ the space of Gevrey-$\beta$ functions.
\end{defi}
\begin{rmk}
    Note that Gevrey-$1$ functions coincide with analytic functions.
\end{rmk}
    
    For fixed $B$, the weighted norm $\|f\|'_{\beta,B}=\sup_{k\in\mathbb N^d}\frac{B^{|k|}}{k!^{\beta}}\|\partial^k f\|_{C_0(\mathbb T^d)}$ seems natural, but it lacks the Banach algebra property (Lemma \ref{Banana^-^}). To address this issue, Marco and Sauzin in \cite{03M.J-Pierre&S.David} first introduced the following alternative norm.
\begin{defi}
    Let $\beta\ge 1$, $R>0$, we define the space of Gevrey-$(\beta,R)$ functions on the $d$-dimensional torus as: 
    $$G_{\beta,R}(\mathbb T^d)=\left\{f\in C^{\infty}(\mathbb T^d):\|f\|_{\beta,R}:=\sum_{k\in\mathbb N^d}\frac{R^{|k|\beta}}{k!^{\beta}}\|\partial^kf\|_{C^0(\mathbb T^d)}<+\infty\right\}.$$
\end{defi}
    One can easily verify that
    $G_{\beta}(\mathbb T^d)=\cup_{R>0}G_{\beta,R}(\mathbb T^d)$.

\begin{defi}
    Given $f\in C^{\infty}(\mathbb T^d)$, for all $r\in\mathbb N$, define 
     \[
     \|f\|_{G_{\beta,R}^r}: =\sum_{j\le r}\|D^jf\|_{\beta,R}, \quad \text{ where }  \displaystyle\|D^rf\|_{\beta,R}:=\max_{|\gamma|_1=r}\|\partial^{\gamma}f\|_{\beta,R}.
     \]
    We define the following subspace $G_{\beta,R}^r(\mathbb T^d)$ of $G_{\beta,R}(\mathbb T^d)$ by
    $$G_{\beta,R}^r(\mathbb T^d):= \left\{f\in C^{\infty}(\mathbb T^d):\|f\|_{G_{\beta,R}^r}< +\infty \right\}.$$
\end{defi}
Note that  $G_{\beta,R}^r(\mathbb T^d)$ is  a Banach space equipped with the norm $\|\cdot\|_{G_{\beta,R}^r}$. Moreover, one can have the following properties.
\begin{lem}[Banach algebra property, \cite{03M.J-Pierre&S.David}]\label{Banana^-^}
    Assume $f,g\in G_{\beta,R}(\mathbb T^d)$, then we have $\|fg\|_{\beta,R}\le\|f\|_{\beta,R}\|g\|_{\beta,R}$.
\end{lem}
\begin{lem}\cite{03M.J-Pierre&S.David}\label{Cauchyest}
    Assume $0<\lambda<R$, $\phi\in G_{\beta,R}(\mathbb T^d)$, then for all $\gamma\in\mathbb N^d$, we have $\partial^\gamma\phi\in G_{\beta,R-\lambda}(\mathbb T^d)$, and $$\sum_{\gamma\in\mathbb N^d}\frac{\lambda^{|\gamma|\beta}}{\gamma!^{\beta}}\|\partial^\gamma\phi\|_{\beta,R-\lambda}\le\|\phi\|_{\beta,R}.$$
    In particular, for all $r\in\mathbb N$, $$\sum_{\gamma\in\mathbb N^d,|\gamma|_1=r}\|\partial^\gamma\phi\|_{\beta,R-\lambda}\le \left(\frac{r!}{\lambda^{r}}\right)^\beta\|\phi\|_{\beta,R}.$$
\end{lem}

\begin{lem}\label{Interpolation ineq}
    Assume $f\in G_{\beta,R}(\mathbb T^d)$, then for all $R_1,R_2 >0$, such that $R^2=R_1R_2$, we have 
    $$\|f\|_{\beta,R}^2\le\|f\|_{\beta,R_1}\|f\|_{\beta,R_2}. $$
\end{lem}
\begin{proof}
    By Cauchy-Schwarz inequality, we have
    \begin{align*}
        \|f\|_{\beta,R}^2&=\left(\sum_{k\in\mathbb N^d}\frac{R^{|k|\beta}}{k!^{\beta}}\|\partial^kf\|_{C^0(\mathbb T^d)}\right)^2\\
        &\le \left(\sum_{k\in\mathbb N^d}\frac{R_1^{|k|\beta}}{k!^{\beta}}\|\partial^kf\|_{C^0(\mathbb T^d)}\right) \left(\sum_{k\in\mathbb N^d}\frac{R_2^{|k|\beta}}{k!^{\beta}}\|\partial^kf\|_{C^0(\mathbb T^d)}\right)\\
        &=\|f\|_{\beta,R_1}\|f\|_{\beta,R_2}.\qedhere
    \end{align*}
\end{proof}

Now, we can define the Gevrey-$(\beta,R)$ quasi-periodic function $f:\mathbb R\to\mathbb R$. 

\begin{defi}\label{rational independent}
$\alpha=(\alpha_1,\cdots,\alpha_d)\in\mathbb R^{d}$ is said to be rationally independent if for all $ (k_1,\cdots,k_d)\in\mathbb Z^d\setminus\{0\} $ we have $\alpha_{1}k_1+\cdots+\alpha_dk_d\ne 0$.  
\end{defi}

\begin{defi}
    The function $f:\mathbb R\to\mathbb R$ is said to be a Gevrey-$(\beta,R)$ quasi-periodic function with basic frequency $\alpha\in\mathbb R^{d}$ and shell function $F: \mathbb T^{d} \to \mathbb R$, if $f(\theta)=F(\alpha\theta)$, $\alpha$ is rationally independent, $F\in G_{\beta,R}(\mathbb T^d)$. And we define the $\beta,R$-norm of $f$ as $\|f\|_{\beta,R}:=\|F\|_{\beta,R}$. We denote by $QP_{\beta,R}(\alpha)$ the set of Gevrey-$(\beta,R)$ quasi-periodic functions with the frequency $\alpha$.
\end{defi}

Since our analysis is restricted to the torus $\mathbb T^d$, it is more convenient to work with the norm defined in terms of Fourier series.
\begin{defi}\cite{17Lopes&Joao}
    Let $\beta\ge1$, $R>0$, we define the function space
    $$\mathcal F_{\beta,R}(\mathbb T^d):= \left\{f(\sigma)=\sum_{k\in\mathbb Z^d}f_ke^{ik\cdot\sigma}\in C^{\infty}(\mathbb T^d):\|f\|_{\mathcal F_{\beta,R}}:=\sum_{k\in\mathbb Z^d}e^{\beta R|k|_{\beta}}| f_k|<+\infty \right\}.$$
\end{defi}
\begin{rmk}
    The space $\mathcal{F}_{\beta,R}(\mathbb T^d)$ endowed with the norm $\|\cdot\|_{\mathcal{F}_{\beta,R}}$ is also a Banach space. Moreover, the two norms $\|\cdot\|_{\beta,R}$, $\|\cdot\|_{\mathcal{F}_{\beta,R}}$ satisfy the following relationship.
\end{rmk}

\begin{defi}
    We denote by $\langle h \rangle$ the average of $h\in C^{\infty}(\mathbb T^d)$,  $$\langle h\rangle:=\frac{1}{(2\pi)^d}\int_{\mathbb T^d}h(\sigma)d\sigma.$$
    And we further define the subspace of $G_{\beta,R_0}(\mathbb T^d)$, $\mathcal{ F}_{\beta,R_0}(\mathbb T^d)$:
    $$\overset{\circ}{G}_{\beta,R}(\mathbb T^d)=\{f\in G_{\beta,R}(\mathbb T^d):\langle f\rangle=0\},$$
    $$\overset{\circ}{\mathcal F}_{\beta,R}(\mathbb T^d)=\{f\in \mathcal F_{\beta,R}(\mathbb T^d):\langle f\rangle=0\}.$$
\end{defi}

\begin{lem}\cite{24You&Yuan}\label{2norm}
    Assume $f\in C^{\infty}(\mathbb T^d)$, then
    \begin{itemize}
    \item [{(1)}] $\|f\|_{\beta,R}\le\|f\|_{\mathcal F_{\beta,R}}$,
    \item [(2)] $\|f\|_{\mathcal F_{\beta,(1-\varepsilon)R-\epsilon}}\le C_{\beta,d}\varepsilon^{-(\beta-1)d}\epsilon^{-\beta d}\|f\|_{\beta,R}$.
    \end{itemize}
\end{lem}

\begin{lem}[Banach algebra property]
    Let $f,g\in \mathcal{F}_{\beta,R}(\mathbb T^d)$, we have $\|fg\|_{\mathcal{F}_{\beta,R}}\le\|f\|_{\mathcal F_{\beta,R}}\|g\|_{\mathcal F_{\beta, R}}.$
\end{lem}
\begin{proof}
    Note that $\displaystyle fg(\sigma)=(\sum_{p\in\mathbb Z^d}f_pe^{ip\cdot\sigma})(\sum_{q\in\mathbb Z^d}g_qe^{iq\cdot\sigma})=\sum_{k\in\mathbb Z^d}\sum_{l\in\mathbb Z^d}f_lg_{k-l}e^{ik\cdot\sigma}$,
    then 
    \begin{align*}
        \|fg\|_{\mathcal{F}_{\beta,R}}&=\sum_{k\in\mathbb Z^d}e^{\beta R|k|_{\beta}}|\sum_{l\in\mathbb Z^d}f_lg_{k-l}|\le\sum_{k\in\mathbb Z^d}\sum_{l\in\mathbb Z^d}e^{\beta R|k|_{\beta}}|f_l||g_{k-l}|\\
        &=\sum_{p\in\mathbb Z^d}\sum_{q\in\mathbb Z^d}e^{\beta R|p+q|_{\beta}}|f_p||g_q|\le\sum_{p\in\mathbb Z^d}\sum_{q\in\mathbb Z^d}e^{\beta R|p|_{\beta}}e^{\beta R|q|_{\beta}}|f_p||g_q|\\
        &=\|f\|_{\mathcal F_{\beta,R}}\|g\|_{\mathcal F_{\beta,R}}.\qedhere
    \end{align*}
\end{proof}

At the end of this section, we present a classical fact on quasi-periodic functions.
\begin{lem}\label{quasilem}
    Let $\alpha\in\mathbb R^d$ be rationally independent. Assume $f:\mathbb R\to\mathbb R$, $f(\theta)=\sum_{k\in\mathbb Z^d}\hat f_ke^{ik\cdot\alpha\theta}$ with $\sum_{k\in\mathbb Z^d}|\hat h_k|<+\infty$. Then, $$\hat f_0=\lim_{T\to\infty}\frac{1}{2T}\int_{-T}^T f(\theta)d\theta.$$
\end{lem}

\subsection{Diophantine properties}
We review some fundamental results about Diophantine properties in the quasi-periodic setting.

The Diophantine vector is $\alpha\in\mathbb R^{d}$ such that
\begin{equation}\label{dio1}
|\alpha\cdot k|\ge\nu_0|k|^{-\tau_0},\,\,\,\forall k\in \mathbb Z^d\setminus\{0\}.
\end{equation} 
We are interested in the numbers $\omega\in\mathbb R$  such that
\begin{equation}\label{dio2}
|\omega\alpha\cdot k-2n\pi|\ge\nu|k|^{-\tau},\,\,\,\forall k\in\mathbb Z^d\setminus\{0\}.
\end{equation}
where $\alpha$ is a Diophantine vector, $\nu_0, \tau_0, \nu,\tau$ are positive numbers.

We denote by $\overline{\mathscr D}(\nu_0, \tau_0)$ the set of $\alpha\in \mathbb R^{d}$ satisfying \eqref{dio1}.
We denote by $\mathscr D(\nu,\tau;\alpha)$ the set of $\omega$ satisfying \eqref{dio2}  and $\mathscr D(\tau;\alpha)=\cup_{\nu>0}\mathscr D(\nu,\tau;\alpha)$.

The following lemmata are classical.

\begin{lem}\cite{09Po}
    For $\nu_0>0$, $\tau_0>d-1$, then $\overline{\mathscr D}(\nu_0, \tau_0)$ is of full Lebesgue measure. Conversely, for $\nu_0>0$, $\tau_0<d-1$, then $\overline{\mathscr D}(\nu_0, \tau_0)=\varnothing$.
\end{lem}

\begin{lem}\cite{12Su&delaLlave}
    If $\alpha\in\overline{\mathscr D}(\nu_0, \tau_0)$, and $\tau>d+\tau_0$, then $\mathscr D(\tau;\alpha)$ is of full Lebesgue measure.
\end{lem}

\subsection{Homological equations}
\begin{lem}\label{homolem}
    Let $\omega\in\mathbb R$, $\alpha\in[0,1]^d$. Assume that $|\omega\alpha\cdot k-2n\pi|\ge\nu|k|^{-\tau}$ for all $k\in \mathbb Z^d\setminus\{0\}$, $n\in\mathbb Z$. If $\eta\in\mathcal F_{\beta,R}$, $\displaystyle\int_{\mathbb T^d}\eta(\sigma) d\sigma=0$, then for all $R'<R$, the equation $$\phi(\sigma+\omega\alpha)-\phi(\sigma)=\eta(\sigma)$$ has the unique solution $\phi$ with $\displaystyle\int_{\mathbb T^d}\phi(\sigma)d\sigma=0$, which satisfies 
    \begin{equation*}
        \|\phi\|_{\mathcal{F}_{\beta,R'}}\le C(\tau,\beta)\nu^{-1}(R-R')^{-{\tau}{\beta}}\|\eta\|_{\mathcal{F}_{\beta,R}}.
    \end{equation*}
\end{lem}
\begin{proof}
    Expanding $\phi$, $\eta$ into Fourier series, we obtain that $\phi_k=\frac{\eta_k}{e^{ik\cdot\omega\alpha}-1}$, for all $k\in\mathbb Z^d\setminus\{0\}$, and $\phi_0=0$. Since $\frac{\pi}{2}|e^{it}-1|\ge{\rm dist}(t,2\pi\mathbb Z)$, it follows that \begin{align*}
        \|\phi\|_{\mathcal F_{\beta,R'}}&=\sum_{k\in\mathbb Z^d}e^{\beta R'|k|_{\beta}}|\phi_k|=\sum_{k\in\mathbb Z^d\setminus\{0\}}e^{\beta R'|k|_{\beta}}|\phi_k|\\
        &\le\sum_{k\in\mathbb Z^d\setminus\{0\}}e^{\beta R'|k|_{\beta}}|\frac{\eta_k}{e^{ik\cdot\omega\alpha}-1}|
        \le\frac{\pi}{2}\sum_{k
        \in\mathbb Z^d\setminus\{0\}}e^{\beta R'|k|_{\beta}}\nu^{-1}|k|_1^{\tau}|\eta_k|\\
        &\le\frac{\pi}{2}\sum_{k\in\mathbb Z^d\setminus\{0\}}e^{\beta R|k|_{\beta}}e^{-\beta(R-R')|k|_{\beta}}|k|_1^{\tau}\nu^{-1}|\eta_k|\\
        &\le\frac{\pi}{2}\sum_{k\in\mathbb Z^d\setminus\{0\}}e^{\beta R|k|_{\beta}}e^{-\beta(R-R')|k|_{\beta}}|k|_\beta^{\beta\tau}\nu^{-1}|\eta_k|.
    \end{align*}
    The last inequality  follows from the fact that $|k|_{\beta}^{\beta}\ge|k|_1$.
    
    Define $f(x)=e^{-(R-R')x}x^{\tau}$, it is easy to see that $f(x)\le f(\frac{\tau}{R-R'})=e^{-\tau}(\frac{\tau}{R-R'})^{\tau}$. Therefore,
    \begin{align*}
        \|\phi\|_{\mathcal F_{\beta,R'}}&\le\frac{\pi}{2}\sum_{k\in\mathbb Z^d\setminus\{0\}}e^{\beta R|k|_{\beta}}e^{-\tau\beta}\tau^{\tau\beta}(R-R')^{-\tau\beta}\nu^{-1}|\eta_k|\\
        &=C(\tau,\beta)\nu^{-1}(R-R')^{-\tau\beta}\|\eta\|_{\mathcal F_{\beta,R}}. \qedhere
    \end{align*}    
\end{proof}

\subsection{Composition}

\begin{lem}\label{composition0}\cite{03M.J-Pierre&S.David}
    Assume $Y\in G_{\beta, R}(\mathbb T^N)$, $u=(u_{[1]},\cdots,u_{[N]})\in(G_{\beta,R_1}(\mathbb T^d))^N$. If 
    \begin{align*}
        \|u_{[1]}\|_{\beta,R_1}-\|u_{[1]}\|_{C^0(\mathbb T^d)},\cdots,\|u_{[N]}\|_{\beta,R_1}-\|u_{[N]}\|_{C^0(\mathbb T^d)}\le\frac{R^\beta}{d^{\beta-1}},
    \end{align*}
    then $Y\circ u\in G_{\beta,R_1}(\mathbb T^d)$ and $\|Y\circ u\|_{\beta,R_1}\le\|Y\|_{\beta,R}$.
\end{lem}

\begin{lem}\label{composition}
    Let $\hat h(\sigma)\in G_{\beta,R}(\mathbb T^d)$, $\widehat H(\zeta_0,\zeta_1,\cdots,\zeta_L)\in G_{\beta,\bar R}^2((\mathbb T^d)^{L+1})$, where $\bar R$ satisfies $\frac{\bar R^{\beta}}{d^{\beta-1}}-dR^{\beta}=\|\hat h\|_{\beta,R}+\iota$, for some $\iota>0$. Define the operator $\Psi$ acting on Gevrey functions by $(\Psi[\hat h])(\sigma)=\widehat H(\sigma+\hat h(\sigma)\alpha,\sigma+\omega\alpha+\hat h(\sigma+\omega\alpha)\alpha,\cdots,\sigma+L\omega\alpha+\hat h(\sigma+L\omega\alpha)\alpha)$, where $\omega\in\mathbb R$.
    
\begin{itemize}
    \item [(i)] For all $R'\le R$, if $\|\hat h^*- \hat h\|_{\beta,R'}\le\iota$, then $\Psi[\hat h^*]\in G_{\beta,R'}(\mathbb T^d)$, and $\|\Psi[\hat h^*]\|_{\beta,R'}\le \|\widehat H\|_{\beta,\bar R}$. \smallskip
    \item [(ii)] For all $R'\le R$, if $\|\hat h^*-\hat h\|_{\beta,R'}\le \iota$, then $\|\Psi[\hat h^*]-\Psi[\hat h]\|_{\beta,R'}\le(L+1)\|D^1\widehat H\|_{{\beta,\bar R}}\|\hat h^*-\hat h\|_{\beta, R'}$. 
    \smallskip
    \item [(iii)] For all $R'\le R$, $(D\Psi[\hat h]\widehat\Delta)(\sigma):=\sum_{j=0}^{L}\partial_{\alpha}^{(j)}\widehat H(\sigma+\hat h(\sigma)\alpha,\cdots,\sigma+L\omega\alpha+\hat h(\sigma+L\omega\alpha)\alpha)(\widehat\Delta(\sigma+j\omega\alpha))$, where $\partial_{\alpha}^{(j)}=\alpha\cdot\partial_j=\alpha\cdot\frac{\partial}{\partial\zeta_j}$,  then \[
\|\Psi[\hat h^*]-\Psi[\hat h]-D\Psi[\hat h](\hat h^*-\hat h)\|_{\beta,R'}\le (L+1)^2 \|D^2\widehat H\|_{{\beta,\bar R}}\|\hat h^*-\hat h\|_{\beta,R'}^2.
\]
\end{itemize}
\end{lem}
\begin{proof}
    (i). Define \[\phi_i:\mathbb T^d\to\mathbb T^d,\,\,\phi_i(\sigma):=\sigma+i\omega\alpha+\hat h^*(\sigma+i\omega\alpha)\alpha,\,\,0\le i\le L,\]
    and \[\phi_i(\sigma):=(\phi_{i,[1]}(\sigma),\cdots,\phi_{i,[d]}(\sigma)),\,\,\text{where}\,\,\phi_{i,[j]}(\sigma)=\sigma_j+i\omega\alpha_j+\hat h^*(\sigma+i\omega\alpha)\alpha_j,\,\,1\le j\le d,\]
    then $\Psi[\hat h^*](\sigma)=\widehat H(\phi_0,\cdots,\phi_L)=\widehat H(\phi_{0,[1]},\cdots,\phi_{0,[d]},\cdots,\phi_{L,[1]},\cdots,\phi_{L,[d]})$. 

    For all $k\in\mathbb N^d\setminus\{0\}$, if $|k|_1=1$, we have $\|\partial^k\phi_{i,[j]}\|_{C^0(\mathbb T^d)}\le1+\|\partial^k\hat h^*\|_{C^0(\mathbb T^d)}$, and if $|k|_1>1$, we have $\|\partial^k\phi_{i,[j]}\|_{C^0(\mathbb T^d)}\le \|\partial^k\hat h^*\|_{C^0(\mathbb T^d)}$.
    
    Therefore, 
    \begin{align*}
        \|\phi_{i,[j]}\|_{\beta,R'}-\|\phi_{i,[j]}\|_{C^0(\mathbb T^d)}&=\sum_{k\in\mathbb N^d\setminus\{0\}}\frac{R'^{|k|_1\beta}}{k!^{\beta}}\|\partial^k\phi_{i,[j]}\|_{C^0(\mathbb T^d)}\le\sum_{|k|_1=1}\frac{R'^{\beta}}{1}+\sum_{|k|_1\ge1}\frac{R'^{|k|\beta}}{k!^{\beta}}\|\partial^k\hat h^*\|_{C^0(\mathbb T^d)}\\
        &=d R'^{\beta}+\|\hat h^*\|_{\beta,R'}\le dR^{\beta}+\|\hat h\|_{\beta,R}+\iota=\frac{\bar R^{\beta}}{d^{\beta-1}}.
    \end{align*}
    By Lemma \ref{composition0}, we have $\Psi[\hat h^*]\in G_{\beta,R'}(\mathbb T^d)$, and $\|\Psi[\hat h^*]\|_{\beta,R'}\le \|\widehat H\|_{\beta,\bar R}$.

    (ii). By (i), we have
    \begin{align*}
        \|\Psi[\hat h^*]-\Psi[\hat h]\|_{\beta, R'}&=\|\int_0^1\sum_{j=0}^L\partial_{\alpha}^{(j)}\widehat H(\sigma+\hat h(\sigma)\alpha+s(\hat h^*(\sigma)-\hat h(\sigma))\alpha,\cdots,\sigma+L\omega\alpha+\hat h(\sigma+L\omega\alpha)\alpha\\&\quad\quad\quad\quad+s(\hat h^*(\sigma+L\omega\alpha)-\hat h(\sigma+L\omega\alpha)))\cdot(\hat h^*(\sigma+j\omega\alpha)-\hat h(\sigma+j\omega\alpha))ds\|_{\beta,R'}\\
        &\le\int_0^1\sum_{j=0}^L\|\partial_{\alpha}^{(j)}\widehat H(\sigma+\hat h(\sigma)\alpha+s(\hat h^*(\sigma)-\hat h(\sigma))\alpha,\cdots,\sigma+L\omega\alpha+\hat h(\sigma+L\omega\alpha)\alpha\\&\quad\quad\quad\quad+s(\hat h^*(\sigma+L\omega\alpha)-\hat h(\sigma+L\omega\alpha)))\|_{\beta,R'}\cdot(\hat h^*(\sigma+j\omega\alpha)-\hat h(\sigma+j\omega\alpha))\|_{\beta,R'}ds\\
        &\le\int_0^1\sum_{j=0}^L\|\partial_{\alpha}^{(j)}\widehat H\|_{\beta,\bar R}\|\hat h^*-\hat h\|_{\beta,R'}ds\le(L+1)\|D^1\widehat H\|_{\beta,\bar R}\|\hat h^*-\hat h\|_{\beta,R}.
    \end{align*}

    (iii). The proof is similar to (ii).
\end{proof}

\section{Long-range interaction model with Gevrey Potential}
\subsection{Hull function and extremal forces}
The Frenkel-Kontorova model describes a chain of particles interacting with their nearest neighbors, where the configuration is represented by a sequence $u=\{u_n\}_{n\in\mathbb Z}$ with $u_n\in\mathbb R$. The formal energy of the system is given by:
\begin{equation}
    \mathscr S(\{u_n\}_{n\in\mathbb Z})=\sum_{i\in\mathbb Z}\sum_{L=0}^{\infty}\widehat H_L(u_i\alpha,u_{i+1}\alpha,\cdots,u_{i+L}\alpha),
\end{equation}
where $\widehat H_L(\zeta_0,\zeta_1,\cdots,\zeta_L)\in G_{\beta,R}((\mathbb T^d)^{L+1})$, $\alpha\in[0,1]^d$ is rationally independent.

The equilibrium configurations of the system are critical points of the formal energy, satisfying {the discrete Euler-Lagrange equation}:
\begin{equation}\label{longrange}
    \frac{\partial\mathscr S}{\partial u_j}=\sum_{L=0}^\infty\sum_{i=j-L}^j\alpha\cdot\partial_{j-i}\widehat H_L(u_i\alpha,\cdots,u_{i+L}\alpha)=0,\,\,\forall j\in\mathbb Z,
\end{equation}
where $\partial_j=\frac{\partial}{\partial\zeta_j}$ for $j=0,1,\cdots,L$. We denote $\partial_{\alpha}^{(j)}=\alpha\cdot\partial_j$ in the following. It is clear that $\partial_{\alpha}^{(j)}\partial_{\alpha}^{(k)}=\partial_{\alpha}^{(k)}\partial_{\alpha}^{(j)}$, for all $j,k\in\mathbb Z$.

Let the rotation number $\omega\in\mathbb R$, and we seek quasi-periodic solutions of the form $$u_n=h(n\omega),\,\,\forall n\in\mathbb Z,$$
where $h(\theta)=\theta+\tilde h(\theta)$, and $\tilde h$ is a Gevrey quasi-periodic function with frequency $\alpha$,  that is, $\tilde h(\theta)=\hat h(\alpha \theta)$, where $\hat h$ is a shell function in $G_{\beta}(\mathbb{T}^d)$.

%The function $\tilde h$ possesses the Fourier series representation,$$\tilde h(\theta)=\sum_{k\in\mathbb Z^d}\hat h_ke^{ik\cdot\alpha\theta}=\hat h(\alpha\theta),$$ 
%where $\hat h$ is a shell function in $G_{\beta}(\mathbb{T}^d)$.

By substituting $u_n=h(n\omega)$ into the equilibrium equation and denoting 
\[
\theta:=j\omega,\,\,\sigma:=\alpha\theta,
\]
\[
h^{(k)}(\theta):=h(\theta+k\omega)=\theta+k\omega+\hat h(\sigma+k\omega\alpha),
\]
\[
 \gamma^{(k)}_L(\theta):=(h^{(k)}(\theta)\alpha,\cdots,h^{(k+L)}(\theta)\alpha)=(\alpha\theta+k\omega\alpha,\hat h(\alpha\theta+k\omega\alpha)\alpha,\cdots,\alpha\theta+(k+L)\omega\alpha+\hat h(\alpha\theta+(k+L)\omega\alpha)\alpha),
\]
\[
\hat\gamma^{(k)}_L(\sigma):=(\sigma+k\omega\alpha+\hat h(\sigma+k\omega\alpha)\alpha,\cdots,\sigma+(k+L)\omega\alpha+\hat h(\sigma+(k+L)\omega\alpha)\alpha)=\gamma_L^{(k)}(\theta),
\]
\[
h(\theta):=h^{(0)}(\theta),\,\,\hat\gamma_L(\sigma):=\hat\gamma_L^{(0)}(\sigma),
\]
we derive the functional equation for $\hat h$:
\begin{equation*}
\sum_{L=0}^\infty\sum_{i=j-L}^j\partial_\alpha^{(j-i)}\widehat H_L(\gamma_L^{(i-j)}(\theta))=\sum_{L=0}^\infty\sum_{k=0}^L\partial_\alpha^{(k)}\widehat H_L(\gamma_L^{(-k)}(\theta))=0.
\end{equation*}
If $\omega\alpha$ satisfies the Diophantine property, then (\ref{longrange}) holds if and only if $\mathscr E[\hat h](\theta)$ defined below vanishes:
\begin{equation}\label{y5}
\begin{aligned}
\mathscr E[\hat h](\theta)&:=\sum_{L=0}^\infty\sum_{k=0}^L\partial_\alpha^{(k)}\widehat H_L(\gamma_L^{(-k)}(\theta))\\
&=\sum_{L=0}^\infty\sum_{k=0}^L\partial_\alpha^{(k)}\widehat H_L(h(\theta-k\omega)\alpha,\cdots,h(\theta-k\omega+L\omega)\alpha)=0,\,\,\forall\theta\in\mathbb R.
\end{aligned}
\end{equation}
We also denote $\mathscr E[\hat h](\sigma)=\sum_{L=0}^\infty\sum_{k=0}^L\partial_\alpha^{(k)}\widehat H_L(\hat\gamma_L^{(-k)}(\sigma))$. If $\omega\alpha$ satisfies the Diophantine condition, then $\mathscr E[\hat h](\sigma)$ vanishes if and only if $\mathscr E[\hat h](\theta)$ vanishes.

%Obviously, if we can solve the general case
%\begin{equation}\label{4}
%\mathcal E[\hat h,\lambda]=\hat h(\sigma+\omega\alpha)+\hat h(\sigma-\omega\alpha)-2\hat h(\sigma)+\widehat U(\sigma+\alpha\hat h(\sigma))+\lambda=0,
%\end{equation}
%then we can easily solve the particular case equation (\ref{3}), just let $\widehat U=\partial_\alpha \widehat V$, $\lambda=0$ and we can get it. Later, we will prove the Vanishing lemma which indicates that if $\widehat U=\partial_\alpha \widehat V$, then $\lambda=0$.

Note that if $\hat h(\sigma)$ is a solution to equation (\ref{longrange}), then for any $\phi\in\mathbb R$, $\hat h(\sigma+\phi\alpha)+\phi$ also satisfies the same equation. Therefore, by choosing $\phi$, we can always assume our solution normalized in such a way that $$\mathop{\rm lim}\limits_{T\rightarrow\infty}\frac1{2T}\int_{-T}^T\tilde h(\theta)d\theta=\frac{1}{(2\pi)^d}\int_{\mathbb T^d}\hat h(\sigma)d\sigma=0.$$

\subsection{Statement of the main theorem}

\begin{thm}\label{main thm}
    Let $h(\theta)=\theta+\tilde h(\theta)$, $\tilde h(\theta)=\hat h(\alpha\theta)=\sum_{k\in\mathbb Z^d}\hat h_ke^{ik\cdot\alpha\theta}$, $\hat h_0=0$, $\hat h\in G_{\beta,R_0}^1(\mathbb T^d)$, where $\beta>1, R_0>0$ 
    and $\alpha\in[0,1]^{d}$ is rationally independent. Denote $\hat l=1+\partial_\alpha\hat h$, $T_{x}(\sigma)=\sigma+x$. We assume the following: 
    \begin{itemize}
        \item[{\rm (H1)}]
        Diophantine condition: $|\omega\alpha\cdot k-2n\pi|\ge\nu|k|_1^{-\tau}$ holds for $\forall k\in\mathbb Z^d\setminus\{0\}$, $\forall n\in\mathbb Z$, where $\tau,\nu>0$.
        \item[{\rm (H2)}]
        Non-degeneracy condition: $\|\hat l(\sigma)\|_{\beta,R_0}\le N^+$, $\|(\hat l(\sigma))^{-1}\|_{\beta,R_0}\le N^-$, $\big|\big<\frac1{\hat l\cdot\hat l\circ T_{-\omega\alpha}}\big>\big|\ge c>0$, where $N^+,N^-,c$ are called the condition numbers. 
        \item[{\rm (H3)}] 
        The interactions $H_L\in G_{\beta,\bar R}^3((\mathbb T^d)^{L+1})$ where $\bar R$ satisfies $\frac{\bar R^{\beta}}{d^{\beta-1}}-dR_0^\beta=\|\hat h\|_{\beta,R_0}+\iota$ for some  $\iota>0$. Denote\[
        M_L=\max_{i=0,1,2,3}(\|D^iH_L\|_{\beta,\bar R}),
        \]
        \[
        \delta=C\sum_{L\ge2}M_LL^4,
        \]
        where C is a combinatorial constant that will be made explicit during the proof.
        \item[{\rm(H4)}]
        Assume that the inverses indicated below exist and have the indicated bounds:
        $$\|(\partial_{\alpha}^{(0)}\partial_{\alpha}^{(1)}\widehat H_1)^{-1}\|_{\beta,\bar R}\le T,$$
        $$\left |\left (\int_{\mathbb T^d}\mathcal C_{0,1,1}^{-1}\right)^{-1}\right|\le U,$$
        where $\mathcal C_{0,1,1}$ is defined in (\ref{y21}).
        \item[{\rm(H5)}]
        $(N^-)^2T\delta<\frac12, (N^-)^2UT<\frac12.$
    \end{itemize}
Assume furthermore that $\|\mathscr E[\hat h]\|_{\beta,R_0}=\epsilon_0\le\epsilon^*(N^+, N^-, c, \tau,\beta,d,\iota, R_0,T,\delta,  U)\nu^4$, where $\epsilon^*>0$ is a function which we will make explicit along the proof.
Then, there exists a Gevrey function $\hat h^*\in G_{\beta,\frac{R_0}{2}}(\mathbb T^d)$ such that $$\mathscr E[\hat h^*]=0.$$
Moreover,$$\|\hat h-\hat h^*\|_{\beta,\frac{R_0}{2}}\le C_1(N^+, N^-, c, \tau, \beta,d, \iota, R_0,T, \delta, U)\nu^{-6}\epsilon_0.$$
The solution $\hat h^*$ is the only solution of $\mathcal{E}[\hat h^*]=0$ with zero average for $\hat h^*$ in a ball centered at $\hat h$ in $G_{\beta,\frac{R_0}{2}}(\mathbb T^d)$, i.e. $\hat h^*$ is the unique solution in the set
$$\{\hat g\in G_{\beta,\frac{R_0}{2}}(\mathbb T^d): \langle \hat g\rangle=0, \|\hat g-\hat h\|_{\beta,\frac{R_0}{2}}\le C_2(N^+, N^-, c, \tau, \beta,d, \iota, R_0,T, \delta, U)\nu^2\}$$
where $C_1$, $C_2$ will be made explicit along the proof.
\end{thm}

\section{Proof of the main theorem}

\subsection{Quasi-Newton iteration}

Applying the quasi-Newton method, we seek the solution of 
\begin{equation}\label{y6}
D\mathscr E[\hat h]\cdot\widehat\Delta=-\mathscr E[\hat h].
\end{equation}
where $D$ denotes the derivative of the functional $\mathscr E$ with respect to $\hat h$. Then $\hat h+\widehat\Delta$ will be a better approximate solution of (\ref{y5}). Unfortunately, equation (\ref{y6}) is hard to solve and we will modify it into the following equation: 
\begin{equation}\label{modeq}
\hat l(D\mathscr E[\hat h]\cdot\widehat\Delta)-\widehat\Delta(D\mathscr E[\hat h]\cdot\hat l)=-\hat l\mathscr E[\hat h],
\end{equation}
where $\hat l=1+\partial_{\alpha}\hat h$.
Equation (\ref{modeq}) is just equation (\ref{y6}) multiplied by $\hat l$ and added the extra term in $\widehat\Delta(D\mathscr E[\hat h]\hat l)$ on the left-hand side. 

Note that 
\begin{equation}\label{y8}
\frac{d}{d\theta}\mathscr E[\hat h](\theta)=(D\mathscr E[\hat h]\cdot\hat l)(\theta)=\sum_{L\in\mathbb N}\sum_{k=0}^L\sum_{j=0}^L\partial_{\alpha}^{(k)}\partial_{\alpha}^{(j)}\widehat H_L(h(\theta-k\omega)\alpha,\cdots,h(\theta-k\omega+L\omega)\alpha)\cdot(1+\partial_{\alpha}\hat h(\sigma-k\omega\alpha+j\omega\alpha)).
\end{equation}
Therefore, we formally obtain
\[
\begin{split}
\mathscr E[\hat h+\widehat\Delta]&=(\mathscr E[\hat h+\widehat\Delta]-\mathscr E[\hat h]-D\mathscr E[\hat h]\widehat\Delta)+\hat l^{-1}(\hat l\mathscr E[\hat h]+\hat l(D\mathscr E[\hat h]\widehat\Delta))\\
&=(\mathscr E[\hat h+\widehat\Delta]-\mathscr E[\hat h]-D\mathscr E[\hat h]\widehat\Delta)+\hat l^{-1}\widehat\Delta(D\mathscr E[\hat h]\hat l)\\
&=(\mathscr E[\hat h+\widehat\Delta]-\mathscr E[\hat h]-D\mathscr E[\hat h]\widehat\Delta)+\hat l^{-1}\widehat\Delta\frac{d}{d\theta}\mathscr E[\hat h].
\end{split}
\]

The standard algorithm to solve (\ref{modeq}) is as follows.

Let $\widehat\Delta=\hat l\cdot\hat\eta$, then the unknowns $\widehat\Delta$ and $\hat\eta$ are equivalent due to the non-degeneracy assumption in Theorem~\ref{main thm}.
Substituting $\widehat\Delta=\hat l\cdot\hat\eta$ into (\ref{modeq}), we obtain 
\begin{equation}\label{y10}
\begin{split}
\sum_{L=0}^\infty\sum_{k=0}^L&\sum_{j=0}^L\partial_\alpha^{(k)}\partial_\alpha^{(j)}\widehat H_L(\hat\gamma_L^{(-k)}(\sigma))\hat l(\sigma)\hat l^{(j-k)}(\sigma)\hat\eta^{(j-k)}(\sigma)\\
&-\sum_{L=0}^\infty\sum_{k=0}^L\sum_{j=0}^L\partial_\alpha^{(k)}\partial_\alpha^{(j)}\widehat H_L(\hat\gamma_L^{(-k)}(\sigma))\hat l(\sigma)\hat l^{(j-k)}(\sigma)\hat\eta(\sigma)\\
=&-\hat l(\sigma)\mathscr E[\hat h](\theta),
\end{split}
\end{equation}
where $\hat l^{(j)}(\sigma)=\hat l(\sigma+j\omega\alpha)$ and $\hat \eta^{(j)}(\sigma)=\hat\eta(\sigma+j\omega\alpha)$.

For fixed $L\in\mathbb{N}$, we note that, when $j=k=0,\cdots,L$ the term in the first sum of the left-hand side of (\ref{y10}) cancels the one in the second sum. When $j\not=k$, we observe that we have four terms involving the mixed derivatives, that is
\begin{equation}\label{y11}
\begin{aligned}
&\partial_\alpha^{(k)}\partial_\alpha^{(j)}\widehat H _L(\hat\gamma_L^{(-k)}(\sigma))\hat l(\sigma)\hat l^{(j-k)}(\sigma)\hat\eta^{(j-k)}(\sigma)\\
+&\partial_\alpha^{(j)}\partial_\alpha^{(k)}\widehat H _L(\hat\gamma_L^{(-j)}(\sigma))\hat l(\sigma)\hat l^{(k-j)}(\sigma)\hat\eta^{(k-j)}(\sigma)\\
-&\partial_\alpha^{(k)}\partial_\alpha^{(j)}\widehat H _L(\hat\gamma_L^{(-k)}(\sigma))\hat l^{(j-k)}(\sigma)\hat l(\sigma)\hat\eta(\sigma)\\
-&\partial_\alpha^{(j)}\partial_\alpha^{(k)}\widehat H _L(\hat\gamma_L^{(-j)}(\sigma))\hat l^{(k-j)}(\sigma)\hat l(\sigma)\hat\eta(\sigma).
\end{aligned}
\end{equation}
We introduce the notations
\begin{equation}\label{y21}
\begin{split}
[\mathcal S_n\hat\eta](\sigma)&:=\hat\eta(\sigma+n\omega\alpha)-\hat\eta(\sigma),\quad\forall n\in\mathbb Z,\hat\eta\in G_{\beta}(\mathbb T^d),\\
\mathcal C_{j,k,L}(\sigma)&:=\partial_\alpha^{(k)}\partial_\alpha^{(j)}\widehat H_L(\hat\gamma_L^{(-k)}(\sigma))\hat l(\sigma)\hat l^{(j-k)}(\sigma).
\end{split}
\end{equation}
With notations above, we can rearrange (\ref{y11}) as
\begin{equation*}
\begin{split}
\mathcal C_{j,k,L}(\sigma)&\cdot[\hat\eta^{(j-k)}-\hat\eta](\sigma)\\
&-\mathcal C_{j,k,L}(\sigma+(k-j)\omega\alpha)[\hat\eta^{(j-k)}-\hat\eta](\sigma+(k-j)\omega\alpha)\\
&=-\mathcal S_{k-j}[\mathcal C_{j,k,L}\mathcal S_{j-k}\hat\eta](\sigma).
\end{split}
\end{equation*}
Therefore, (\ref{y10}) can be written as
\begin{equation}\label{y12}
    \sum_{L=1}^\infty\sum_{0\le j<k\le L}\mathcal S_{k-j}[\mathcal C_{j,k,L}\mathcal S_{j-k}\hat\eta](\sigma)=\hat l(\sigma)\mathscr E[\hat h](\theta).
\end{equation}

Now we study the invertibility of the operators $\mathcal S_n$. In fact, $\mathcal S_n:\mathcal F_{\beta,R}(\mathbb T^d)\to\overset{\circ}{\mathcal F}_{\beta,R}(\mathbb T^d)$ is diagonal on Fourier series. Due to the Diophantine condition, by Lemma \ref{homolem} for any given $\hat\eta\in\overset{\circ}{\mathcal F}_{\beta,R}(\mathbb T^d)$ and $R'<R$, we can find the solution of \[\mathcal S_n\hat \gamma=\hat\eta,\] where $\hat\gamma\in\mathcal F_{\beta,R'}(\mathbb T^d)$, and the solutions are unique up to additive constants. Therefore, for all $R'<R$, the operator $\mathcal S_n^{-1}:\overset{\circ}{\mathcal F}_{\beta,R}(\mathbb T^d)\to\overset{\circ}{\mathcal F}_{\beta,R'}(\mathbb T^d)$ is well-defined. By Lemma \ref{2norm} the operators $\mathcal S_n:G_{\beta,R}(\mathbb T^d)\to\overset{\circ}{G}_{\beta,R}(\mathbb T^d)$ and $\mathcal S_n^{-1}:\overset{\circ}{G}_{\beta,R}(\mathbb T^d)\to\overset{\circ}{G}_{\beta,R'}(\mathbb T^d)$ are also well-defined. 

Hence, we can define the operators $$\mathcal L_n^{\pm}:=\mathcal S^{-1}_{\pm1}\mathcal S_n: G_{\beta,R}(\mathbb T^d)\to\overset{\circ}{G}_{\beta,R}(\mathbb T^d),$$
$$\mathcal R_n^{\pm}:=\mathcal S_n\mathcal S^{-1}_{\pm1}:\overset{\circ}{G}_{\beta,R}(\mathbb T^d)\to\overset{\circ}{G}_{\beta,R}(\mathbb T^d).$$
And we have the following Lemma.
\begin{lem}
    For all $n\in\mathbb Z$,  we have
    \[
    \begin{split}
    &\|\mathcal L_n^\pm\|_{\mathscr L(G_{\beta,R}(\mathbb T^d),\overset{\circ}{G}_{\beta,R}(\mathbb T^d))}\le|n|,\\
    &\|\mathcal R_n^\pm\|_{\mathscr L(\overset{\circ}{G}_{\beta,R}(\mathbb T^d),\overset{\circ}{G}_{\beta,R}(\mathbb T^d))}\le|n|,
    \end{split}
    \]
    where $\mathscr L(X,Y)$ denotes the space of all bounded linear operators from $X$ to $Y$.
\end{lem}
\begin{proof}
    Denoting by $\hat\phi:=\mathcal L_n^{+}\hat\eta=\mathcal S_1^{-1}\mathcal S_n\hat\eta$, $n>0$, then \[\hat\phi(\sigma+\omega\alpha)-\hat\phi(\sigma)=\hat\eta(\sigma+n\omega\alpha)-\hat\eta(\sigma).\]
    By the uniqueness, we have \[
    \hat\phi(\sigma)=\hat\eta(\sigma+(n-1)\omega\alpha)+\cdots+\hat\eta(\sigma)-n\int_{\mathbb T^d}\hat\eta(\sigma)d\sigma.
    \]
    Therefore, \[
    \begin{aligned}
       \|\mathcal L_n^{+}\|_{\mathscr L(G_{\beta,R}(\mathbb T^d),\overset{\circ}{G}_{\beta,R}(\mathbb T^d))}&=\sup_{\|\hat\eta\|_{\beta,R}\le1}\|\mathcal L_n^+\hat\eta\|_{\beta,R}\\&\le\sup_{\|\hat\eta\|_{\beta,R}\le1}\|\hat\eta(\sigma+(n-1)\omega\alpha)+\cdots+\hat\eta(\sigma)-n\int_{\mathbb T^d}\hat\eta(\sigma)d\sigma\|_{\beta,R}\\
       &\le\sup_{\|\hat\eta\|_{\beta,R}\le1}\|\hat\eta(\sigma+(n-1)\omega\alpha)-\int_{\mathbb T^d}\hat\eta(\sigma)d\sigma\|_{\beta,R}+\cdots+\|\hat\eta(\sigma)-\int_{\mathbb T^d}\hat\eta(\sigma)d\sigma\|_{\beta,R}\\
       &\le\sup_{\|\hat\eta\|_{\beta,R}\le1}\|\hat\eta(\sigma+(n-1)\omega\alpha)\|_{\beta,R}+\cdots+\|\hat\eta(\sigma)\|_{\beta,R}\\
       &\le n.
    \end{aligned}
    \]

    Denoting by $\hat\varphi:=\mathcal S_{1}^{-1}\hat\eta$, then \[\hat\varphi(\sigma+\omega\alpha)-\hat\varphi(\sigma)=\hat\eta(\sigma).\]
    For $n>0$, we have \[\mathcal R_n^+\hat\eta=\mathcal S_n\mathcal S_1^{-1}\hat\eta=\mathcal S_n\hat\varphi(\sigma)=\hat\varphi(\sigma+n\omega\alpha)-\hat\varphi(\sigma)=\hat\eta(\sigma+(n-1)\omega\alpha)+\cdots+\hat\eta(\sigma).\]
    Therefore
    $
        \|\mathcal R_n^+\|_{\mathscr L(\overset{\circ}{G}_{\beta,R}(\mathbb T^d),\overset{\circ}{G}_{\beta,R}(\mathbb T^d))}\le n.
    $

    The other cases are proved in the same way.
\end{proof}

Therefore, (\ref{y12}) can be written as
\begin{equation}\label{y13}
\begin{split}
\hat l(\sigma)\mathscr E[\hat h](\sigma)&=\mathcal S_1[\mathcal C_{0,1,1}\mathcal S_{-1}\hat\eta](\sigma)+\sum_{L=2}^\infty\sum_{0\le j<k\le L}\mathcal S_{k-j}[\mathcal C_{j,k,L}\mathcal S_{j-k}\hat\eta](\sigma)\\
&=\mathcal S_1[\mathcal C_{0,1,1}+\sum_{L=2}^\infty\sum_{0\le j<k\le L}\mathcal S_1^{-1}\mathcal S_{k-j}\mathcal C_{j,k,L}\mathcal S_{j-k}\mathcal S_{-1}^{-1}]\mathcal S_{-1}\hat\eta(\sigma)\\
&=:\mathcal S_1[\mathcal C_{0,1,1}+\mathscr G]\mathcal S_{-1}\hat\eta(\sigma).
\end{split}
\end{equation}
We denote
\begin{equation*}
\mathscr G:=\sum_{L=2}^\infty\sum_{0\le j<k\le L}\mathcal S_1^{-1}\mathcal S_{k-j}\mathcal C_{j,k,L}\mathcal S_{j-k}\mathcal S_{-1}^{-1}=\sum_{L=2}^\infty\sum_{0\le j<k\le L}\mathcal L_{k-j}^+\mathcal C_{j,k,L}\mathcal R_{j-k}^-.
\end{equation*}

The equation(\ref{y13}) is equivalent to the following system:
\begin{equation}\label{homo1}
\mathcal S_1\widehat W(\sigma)=\hat l(\sigma)\mathscr E[\hat h](\theta),
\end{equation}
\begin{equation}\label{homo2}
\mathcal S_{-1}[\hat\eta](\sigma)=[\mathcal C_{0,1,1}+\mathscr G]^{-1}\widehat W(\sigma).
\end{equation}
It is easy to see that
\[
\begin{split}
\int_{\mathbb{T}^d}\hat l(\sigma)\mathscr E[\hat h]d\sigma&=\sum_{L\in\mathbb N}\sum_{k=0}^{L}\int_{\mathbb{T}^d}\hat l(\sigma)\partial_\alpha^{(k)}\widehat H_L(\hat\gamma_L^{(-k)}(\sigma))d\sigma\\
&=\sum_{L\in\mathbb N}\sum_{k=0}^{L}\int_{\mathbb{T}^d}\partial_\alpha^{(k)}\widehat H_L(\hat\gamma_L(\sigma))\hat l^{(k)}(\sigma)d\sigma\\
&=\sum_{L\in\mathbb N}\int_{\mathbb{T}^d}d\widehat H_L(\hat\gamma_L(\sigma))=0.
\end{split}
\]

Express $\widehat W=\widehat W^0+\overline{\widehat W}$, where $\widehat W^0$ is a function with zero average and $\overline{\widehat W}$ is the average of $\widehat W$. Applying Lemma \ref{homolem} to (\ref{homo1}), we get $\widehat W^0$.

Since the left-hand side of equation (\ref{homo2}) has zero average, we need the right-hand side to also have zero average.
Therefore
\[
\int_{\mathbb T ^{d}}(\mathcal C_{0,1,1}+\mathscr G)^{-1}[{\widehat W}]d\sigma=0.
\]
Applying Lemma \ref{homolem} to (\ref{homo2}), we get $\hat\eta$. Then $\widehat\Delta=\hat l\cdot\hat\eta$ is the solution of (\ref{modeq}).

\subsection{Algorithm}
%The solution procedure for (\ref{y13}) consists of the following steps:
The procedure to improve an approximate solution consists of the following steps:
\begin{itemize}
    \item [{\rm (1)}] Solve the cohomology equation (\ref{homo1}) for $\widehat W^0$ with zero average.
    \item [{\rm (2)}] Take $\overline{\widehat W}$ such that 
    \[
        \int_{\mathbb T^d}(\mathcal C_{0,1,1}+\mathscr G)^{-1}[\widehat W^0+\overline{\widehat W}]d\sigma=0.
    \]
    
    \item [{\rm (3)}] Calculate $\widehat W=\widehat W^0+\overline{\widehat W}$.
    \item [{\rm (4)}] Solve the cohomology equation (\ref{homo2}).
    \item [{\rm (5)}] Finally, we obtain the improved solution:
    \[
    \hat h+\widehat\Delta=\hat h+\hat l\cdot \hat\eta.
    \]

\end{itemize}

\subsection{Estimates for one iterative step}
%\paragraph{\textbf{Estimates for approximate solutions}}
\subsubsection{Estimates for approximate solutions}
For all $\varepsilon,\epsilon>0$, $R_0\ge R>0$, by Lemma \ref{2norm}, for all $f\in G_{\beta,R}(\mathbb T^d)$, we have \[\|f\|_{\mathcal{F}_{\beta,(1-\varepsilon)R-\epsilon}}\le C_{\beta,d}\varepsilon^{-(\beta-1)d}\epsilon^{-\beta d}\|f\|_{\beta,R}.\] 

%For convenience, we denote $J=C_{\beta,d}\varepsilon^{-(\beta-1)d}\epsilon^{-\beta d}$.
Denote $R'=(1-\varepsilon_1)R-\epsilon_1$. Applying Lemma \ref{homolem} to the equation (\ref{homo1}), for all $\frac{R_0}{2}<R''<R'<R$, we have 
\begin{align*}
    \|\widehat W^0\|_{\beta,R''}\le\|\widehat W^0\|_{\mathcal F_{\beta,R''}}&\le C(\tau,\beta)\nu^{-1}(R'-R'')^{-\tau\beta}\|\hat l\cdot\mathscr E[\hat h]\|_{\mathcal F_{\beta,R'}}\\
    &\le C(\tau,\beta)\nu^{-1}(R'-R'')^{-\tau\beta}C_{\beta,d}\varepsilon_1^{-(\beta-1)d}\epsilon_1^{-\beta d}\|\hat l\cdot\mathscr E[\hat h]\|_{\beta,R}\\
    &\le C(\tau,\beta)\nu^{-1}(R'-R'')^{-\tau\beta}C_{\beta,d}\varepsilon_1^{-(\beta-1)d}\epsilon_1^{-\beta d}N^+\|\mathscr E[\hat h]\|_{\beta,R}.
\end{align*}

Since $\widehat W=\widehat W^0+\overline{\widehat W}$, the equation for $\overline{\widehat W}$ can be written as 
\begin{equation*}
    \int_{\mathbb T^d}\mathcal C_{0,1,1}^{-1}\overline{\widehat W}d\sigma+\int_{\mathbb T^d}[(\mathcal C_{0,1,1}+\mathscr G)^{-1}-\mathcal C_{0,1,1}^{-1}][\widehat W]d\sigma=-\int_{\mathbb T^d}\mathcal C_{0,1,1}^{-1}\widehat W^0d\sigma.
\end{equation*}
Thus, we have
\begin{equation*}
    |\overline{\widehat W}|\le U\|(\mathcal C_{0,1,1}+\mathscr G)^{-1}\widehat W-\mathcal C_{0,1,1}^{-1}\widehat W\|_{\beta,R''}+U(N^-)^2T\|\widehat W^0\|_{\beta,R''}.
\end{equation*}

Due to the assumption (\rm{H2}) and (\rm{H3}) in Theorem \ref{main thm}, we obtain the following estimate:
\begin{align*}
    \|\mathscr G\|_{\mathscr L(\overset{\circ}{G}_{\beta,R''},\overset{\circ}{G}_{\beta,R''})}&=\sup_{\substack{\|\hat\eta\|_{\beta,R''}=1,\\\hat\eta\in\overset{\circ}{G}_{\beta,R''}}}\le\sum_{L=2}^{\infty}\sum_{0\le j<k\le L}\|\mathcal L_{k-j}^+\mathcal C_{j,k,L}\mathcal R_{j-k}^-\hat\eta\|_{\beta,R''}\\
    &\le\sum_{L=2}^{\infty}\sum_{0\le j<k\le L}|k-j|^2M_L(N^+)^2\\
    &\le C\sum_{L=2}^{\infty}M_LL^4=:\delta.
\end{align*}

By (\rm{H5}) the usual Neumann series shows that the operator $\mathcal C_{0,1,1}+\mathscr G$ is boundedly invertible from $\mathcal G_{\beta,R''}$ to $G_{\beta,R''}$. Moreover, we have
\begin{equation*}\label{y14}
\begin{aligned}
   \|(\mathcal C_{0,1,1}+\mathscr G)^{-1}-\mathcal C_{0,1,1}^{-1}\|_{\mathscr L(G_{\beta,R''},G_{\beta,R''})}&=\|[\mathcal C_{0,1,1}(Id+\mathcal C^{-1}_{0,1,1}\mathscr G)]^{-1}-\mathcal C_{0,1,1}^{-1}\|_{\mathscr L(G_{\beta,R''},G_{\beta,R''})}&\\
&=\|\sum_{j=0}^\infty(-\mathcal C^{-1}_{0,1,1}\mathscr G)^j\mathcal C_{0,1,1}^{-1}-\mathcal C_{0,1,1}^{-1}\|_{\mathscr L(G_{\beta,R''},G_{\beta,R''})}\\
&\le\|\mathcal C_{0,1,1}^{-1}\|_{\beta,R''}\sum_{j=1}^\infty\|\mathcal C_{0,1,1}^{-1}\mathscr G\|^j_{\mathscr L(G_{\beta,R''},G_{\beta,R''})}\\
&\le(N^-)^2T\frac{(N^-)^2T\delta}{1-(N^-)^2T\delta}\\
&\le(N^-)^2T. 
\end{aligned}
\end{equation*}
By assumption (\rm{H5}), we obtain
\begin{align*}
    |\overline{\widehat W}|\le U(N^-)^2T(\|\widehat W^0\|_{\beta,R''}+\overline{\widehat W})+UT(N^-)^2\|\widehat W^0\|_{\beta,R''}\le\|\widehat W^0\|_{\beta,R''}+\frac12|\overline{\widehat W}|.
\end{align*}
Hence, $|\overline{\widehat W}|\le 2\|\widehat W^0\|_{\beta,R''}$ and $\|\widehat W\|_{\beta,R''}\le 3\|\widehat W^0\|_{\beta,R''}$.

Denote $R'''=(1-\varepsilon_2)R''-\epsilon_2$. Applying Lemmata \ref{2norm} and \ref{homolem} to the equation (\ref{homo2}), for all $\frac{R_0}{2}< R^{(4)}<R'''$, we have
\begin{align*}
    \|\hat\eta\|_{\beta,R^{(4)}}&\le\|\hat\eta\|_{\mathcal{F}_{\beta,R^{(4)}}}\le C(\tau,\beta)\nu^{-1}(R'''-R^{(4)})^{-\tau\beta}\|(\mathcal C_{0,1,1}+\mathscr G)^{-1}\widehat W\|_{\mathcal{F}_{\beta,R'''}}\\
    &\le C(\tau,\beta)\nu^{-1}(R'''-R^{(4)})^{-\tau\beta}C_{\beta,d}\varepsilon_2^{-(\beta-1)d}\epsilon_2^{-\beta d}\|(\mathcal C_{0,1,1}+\mathscr G)^{-1}\widehat W\|_{\beta,R''}\\
    &\le C(\tau,\beta)\nu^{-1}(R'''-R^{(4)})^{-\tau\beta}C_{\beta,d}\varepsilon_2^{-(\beta-1)d}\epsilon_2^{-\beta d}6(N^-)^2T\|\widehat W^0\|_{\beta,R''}.
\end{align*}
Since $\widehat\Delta=\hat l\cdot\hat\eta$, it follows that for all $\frac{R_0}{2}<R^{(5)}<R^{(4)}$, we have
\begin{equation*}\label{y19a}
    \begin{aligned}
    \|\widehat\Delta\|_{\beta,R^{(5)}}\le&\|\widehat\Delta\|_{\beta,R^{(4)}}\le N^+\|\hat\eta\|_{\beta,R^{(4)}}\\
    \le& N^+ C(\tau,\beta)\nu^{-1}(R'''-R^{(4)})^{-\tau\beta}C_{\beta,d}\varepsilon_2^{-(\beta-1)d}\epsilon_2^{-\beta d}6(N^-)^2T\\
    &\cdot C(\tau,\beta)\nu^{-1}(R'-R'')^{-\tau\beta}C_{\beta,d}\varepsilon_1^{-(\beta-1)d}\epsilon_1^{-\beta d}N^+\|\mathcal E[\hat h]\|_{\beta,R}.
    \end{aligned}
\end{equation*}
By Lemma \ref{Cauchyest},  we have
\begin{align*}
    \|\partial_{\alpha}\widehat\Delta\|_{\beta,R^{(5)}}\le(R^{(4)}-R^{(5)})^{-\beta}\|\widehat\Delta\|_{\beta,R^{(4)}}.
\end{align*}

By composition Lemma \ref{composition}, if $\|\widehat\Delta\|_{\beta,R^{(5)}}\le\iota$, then $\mathscr E[\hat h+\widehat\Delta]$ is well-defined and
\begin{equation*}
    \begin{aligned}
        \|\mathscr E[\hat h+\hat\Delta]-\mathscr E[\hat h]-D\mathscr E[\hat h]\widehat\Delta\|_{\beta,R^{(5)}}\le\sum_{L=0}^{\infty}\sum_{k=0}^L(L+1)^2M_L\|\widehat\Delta\|_{\beta,R^{(5)}}^2=\sum_{L=0}^\infty(L+1)^3M_L\|\widehat\Delta\|_{\beta,R^{(5)}}^2.
    \end{aligned}
\end{equation*}
Since $\frac{d}{d\theta}\mathscr E[\hat h](\theta)=\partial_{\alpha}\mathscr E[\hat h](\sigma)$, by Lemma \ref{Cauchyest}, 
\[
\|\frac{d}{d\theta}\mathscr E[\hat h]\|_{\beta,R^{(5)}}\le (R-R^{(5)})^{-\beta}\|\mathscr E[\hat h]\|_{\beta,R}.
\]
Therefore,
\[
\begin{aligned}
    \|\mathscr E[\hat h+\widehat\Delta]\|_{\beta,R^{(5)}}&=\|\mathscr E[\hat h+\widehat\Delta]-\mathscr E[\hat h]-D\mathscr E[\hat h]\widehat\Delta+\frac{\widehat\Delta}{\hat l}\frac{d}{d\theta}\mathscr E[\hat h]\|_{\beta,R^{(5)}}\\
    &\le\sum_{L=0}^\infty(L+1)^3M_L\|\widehat\Delta\|_{\beta,R^{(5)}}^2+N^-\|\widehat\Delta\|_{\beta,R^{(5)}}(R-R^{(5)})^{-\beta}\|\mathscr E[\hat h]\|_{\beta,R}.
\end{aligned}
\]

Let $\kappa=R-R^{(5)}$, and take $\varepsilon_1=\frac{\kappa}{8R}$, $\epsilon_1=\frac{\kappa}{8}$, $R'-R''=\frac{\kappa}{8}$, $\varepsilon_2=\frac{\kappa}{8R''}$, $\epsilon_2=\frac{\kappa}{8}$, $R'''-R^{(4)}=\frac{\kappa}{8}$.
For ease of notation, in the following calculations, we set $R_0<1$, $\nu<1$.
Therefore, if $\|\widehat\Delta\|_{\beta,R^{(5)}}\le\iota$, we have the following estimates for one iterative step:
\begin{equation}\label{estDelta}
    \begin{aligned}
        \|\widehat\Delta\|_{\beta,R^{(5)}}\le\|\widehat\Delta\|_{\beta,R^{(4)}}\le&6(N^+)^2(N^-)^2TC(\tau,\beta)^2\nu^{-2}C_{\beta,d}^2(\frac{\kappa}{8})^{-2\tau\beta}\\
        &\cdot(\frac{\kappa}{8R})^{-(\beta-1)d}(\frac{\kappa}{8})^{-\beta d}(\frac{\kappa}{8R''})^{-(\beta-1)d}(\frac{\kappa}{8})^{-\beta d}\|\mathscr E[\hat h]\|_{\beta,R}\\
        \le& C_1\nu^{-2}\kappa^{-2\tau\beta-2(\beta-1)d-2\beta d}\|\mathscr E[\hat h]\|_{\beta, R},
    \end{aligned}
\end{equation}
where $C_1=C_1(N^+,N^-,c,\tau,\beta,d,R_0,T)$. 
%$C_2=2\tau\beta+2(\beta-1)d+2\beta d$
\begin{equation}\label{estDeltaderivative}
    \begin{aligned}
        \|\partial_{\alpha}\widehat\Delta\|_{\beta,R^{(5)}}&\le(\frac{\kappa}{4})^{-\beta}C_1\nu^{-2}\kappa^{-2\tau\beta-2(\beta-1)d-2\beta d}\|\mathscr E[\hat h]\|_{\beta, R}\\
        &\le 4^{\beta}C_1\nu^{-2}\kappa^{-\beta-2\tau\beta-2(\beta-1)d-2\beta d}\|\mathscr E[\hat h]\|_{\beta,R}.
    \end{aligned}
\end{equation}
\begin{equation}\label{estsol}
    \begin{aligned}
        \|\mathscr E[\hat h+\widehat\Delta]\|_{\beta,R^{(5)}}\le&\sum_{L=0}^{\infty}(L+1)^3M_L\big(C_1\nu^{-2}\kappa^{-2\tau\beta-2(\beta-1)d-2\beta d}\|\mathscr E[\hat h]\|_{\beta, R}\big)^2\\&+N^-\kappa^{-\beta}C_1\nu^{-2}\kappa^{-2\tau\beta-2(\beta-1)d-2\beta d}\|\mathscr E[\hat h]\|^2_{\beta, R}\\
        \le& C_2 \nu^{-4}\kappa^{-C_3}\|\mathscr E[\hat h]\|^2_{\beta,R},
    \end{aligned}
\end{equation}
where $C_2=C_2(N^+,N^-,c,\tau,\beta,d,R_0,T,\delta)$,
$C_3=C_3(\tau,\beta,d)=\beta+4\tau\beta+4(\beta-1)d+4\beta d$.

%\paragraph{\textbf{Estimates for the condition numbers}}
\subsubsection{Estimates for the condition numbers}
Let us now estimate the change of non-degeneracy conditions in the iterative step. We denote by $\tilde\gamma_L$ the one corresponding to $\hat h+\widehat\Delta$ instead of $\hat h$. Using the Cauchy estimate and the mean value theorem, we have
\begin{equation*}
    \begin{aligned}
        \|\partial_\alpha^{(0)}\partial_\alpha^{(1)}\widehat H_1(\tilde\gamma_1(\sigma))-\partial_\alpha^{(0)}\partial_\alpha^{(1)}\widehat H_1(\gamma_1(\sigma))\|_{\beta,R^{(5)}}\le2M_1\|\widehat\Delta\|_{\beta,R^{(5)}}\le 2M_1 C_1\nu^{-2}\kappa^{-\frac{C_3}{2}}\|\mathscr E[\hat h]\|_{\beta,R}.
    \end{aligned}
\end{equation*}

We define 
\[C_1\nu^{-2}\kappa^{-\frac{C_3}{2}}\|\mathscr E[\hat h]\|_{\beta,R}=:\chi,\,\,4^{\beta}C_1\nu^{-2}\kappa^{-\frac{C_3}{2}-\beta}\|\mathscr E[\hat h]\|_{\beta,R}=:\chi',\]
then we have $\|\widehat\Delta\|_{\beta,R^{(5)}}\le\chi$, $\|\partial_{\alpha}\widehat\Delta\|_{\beta,R^{(5)}}\le\chi'$.

If $\chi'\le(N^+)^2+N^+$ and $\frac{\kappa}{4}\le1$, then $\chi\le\chi'$ and we obtain:
\begin{equation*}
\begin{split}
\|\widetilde{\mathcal C}_{0,1,1}-\mathcal C_{0,1,1}\|_{\beta,R^{(5)}}&=\|\partial_\alpha^{(0)}\partial_\alpha^{(1)}\widehat H_1(\tilde\gamma_1^{(-1)}(\sigma))(\hat l+\partial_\alpha\widehat\Delta)(\sigma)(\hat l+\partial_\alpha\widehat\Delta)(\sigma-\omega\alpha)\\
&-\partial_\alpha^{(0)}\partial_\alpha^{(1)}\widehat H_1(\hat\gamma_1^{(-1)}(\sigma))\hat l(\sigma)\hat l(\sigma-\omega\alpha)\|_{\beta,R^{(5)}}\\
&\le\|\partial_\alpha^{(0)}\partial_\alpha^{(1)}\widehat H_1(\tilde\gamma_1^{(-1)}(\sigma))-\partial_\alpha^{(0)}\partial_\alpha^{(1)}\widehat H_1(\hat\gamma_1^{(-1)}(\sigma))\|_{\beta,R^{(5)}}\|\hat l(\sigma)\hat l(\sigma-\omega\alpha)\|_{\beta,R^{(5)}}\\
&+\|\partial_\alpha^{(0)}\partial_\alpha^{(1)}\widehat H_1(\tilde\gamma_1^{(-1)}(\sigma))\hat l(\sigma)\partial_\alpha\widehat\Delta(\sigma-\omega\alpha)\|_{\beta,R^{(5)}}\\
&+\|\partial_\alpha^{(0)}\partial_\alpha^{(1)}\widehat H_1(\tilde\gamma_1^{(-1)}(\sigma))\partial_\alpha\widehat\Delta(\sigma)\hat l(\sigma-\omega\alpha)\|_{\beta,R^{(5)}}\\
&+\|\partial_\alpha^{(0)}\partial_\alpha^{(1)}\widehat H_1(\tilde\gamma_1^{(-1)}(\sigma))\partial_\alpha\widehat\Delta(\sigma)\partial_\alpha\widehat\Delta(\sigma-\omega\alpha)\|_{\beta,R^{(5)}}\\
&\le2M_1\chi(N^+)^2+2M_1N^+\chi'+M_1(\chi')^2\\
&\le3M_1\chi'((N^+)^2+N^+).
\end{split}
\end{equation*}

We use the same notations as in Theorem \ref{main thm}, but use the $\sim$ to indicate that they are estimated at the function $\hat h+\widehat\Delta$. Therefore, it is easy to check by the mean value theorem and Cauchy estimates:
\begin{equation}\label{y19}
\begin{split}
%\widetilde T&\equiv\|(\partial_\alpha^{(0)}\partial_\alpha^{(1)}\widehat H)^{-1}(\tilde\gamma_1(\sigma))\|_{\rho'}\\
%&\le T+\|(\partial_\alpha^{(0)}\partial_\alpha^{(1)}\widehat H_1)^{-1}(\tilde\gamma_1(\sigma))-(\partial_\alpha^{(0)}\partial_\alpha^{(1)}\widehat H_1)^{-1}(\gamma_1(\sigma))\|_{\rho''}\\
%&\le T+\bar\sigma^{-1}T\chi,\\
%\widetilde N^+&\equiv\|1+\partial_\alpha(\hat h+\widehat\Delta)\|_{\rho'}\le N^++\|\partial_\alpha\widehat\Delta\|_{\rho'}\le N^++\chi',\\
\widetilde N^+&:=\|1+\partial_\alpha(\hat h+\widehat\Delta)\|_{\beta,R^{(5)}}\le N^++\|\partial_{\alpha}\widehat\Delta\|_{\beta,R^{(5)}}   \\
\widetilde N^-&:=\|(1+\partial_\alpha(\hat h+\widehat\Delta))^{-1}\|_{\beta,R^{(5)}}\le N^-+\sum_{j=1}^{\infty}\|(-\frac{\partial_\alpha\widehat\Delta}{1+\partial_\alpha\hat h})^j(1+\partial_\alpha\hat h)^{-1}\|_{\beta,R^{(5)}}\\
&\le N^-+\sum_{j=1}^{\infty}(N^-\chi')^jN^-=N^-+\frac{\chi'(N^-)^2}{1-\chi'N^-},\\
\widetilde U^{-1}&:=\left |\int_{\mathbb T^d}\widetilde{\mathcal C}_{0,1,1}^{-1}\right |=\left |\int_{\mathbb T^{d}}\left \{\mathcal C_{0,1,1}[Id+\mathcal C_{0,1,1}^{-1}(\widetilde{\mathcal C}_{0,1,1}-\mathcal C_{0,1,1})]\right \}^{-1}\right |\\
&\ge U^{-1}(1-\sum_{j=1}^{\infty}\|\mathcal C_{0,1,1}^{-1}(\widetilde{\mathcal C}_{0,1,1}-\mathcal C_{0,1,1})\|_{\beta,R^{(5)}}^j)\\
&\ge U^{-1}-\frac{3(N^-)^2TM_1((N^+)^2+N^+)\chi'}{1-3(N^-)^2TM_1((N^+)^2+N^+)\chi'}U^{-1},\\
|\tilde c-c|&:=\left |\langle\frac1{(\hat l+\partial_{\alpha}\widehat{\Delta})\cdot(\hat l+\partial_{\alpha}\widehat{\Delta})\circ T_{-\omega\alpha}}\rangle-\langle\frac1{\hat l\cdot\hat l\circ T_{-\omega\alpha}}\rangle\right |\\
&=\left |\langle\frac{(\hat l\cdot\partial_{\alpha}\widehat\Delta\circ T_{-\omega\alpha}+\partial_{\alpha}\widehat\Delta\cdot(\hat l+\partial_{\alpha}\widehat\Delta)\circ T_{-\omega\alpha})}{(\hat l+\partial_{\alpha}\widehat{\Delta})\cdot(\hat l+\partial_{\alpha}\widehat{\Delta})\circ T_{-\omega\alpha}\cdot\hat l\cdot\hat l\circ T_{-\omega\alpha}}\rangle\right |\\
&\le(\widetilde N^-)^2(N^-)^2\chi'(2N^++\chi').
\end{split}
\end{equation}

\subsection{Convergence for the whole procedure}
Let $ R_n=R_{n-1}-\frac{R_0}{2^2}2^{-n}=R_0(1-\frac{1}{2^2}\sum_{i=1}^n2^{-i})$, then $R_n\to R_{\infty}=\frac{3}{4}R_0$, as $n\to\infty$.

We denote $\hat h_0:=\hat h$. In the first iterative step, we obtain $\widehat\Delta_0$. We define $\hat h_1=\hat h_0+\widehat\Delta_0$. Similarly, we define $\hat h_n=\hat h_{n-1}+\widehat\Delta_{n-1}$ where $\widehat\Delta_{n-1}\in G_{\beta,R_n}(\mathbb T^d)$. Denote $\epsilon_n=\|\mathscr E[\hat h_n]\|_{\beta,R_n}$. We will prove that in the iterative steps, $\mathscr E[\hat h_n]$ is well-defined. Moreover, $N^{+}(\hat h_n,R_n):=\|1+\partial_{\alpha}\hat h_n\|_{\beta,R_n}\le 2N^{+}$, $N^{-}(\hat h_n,R_n):=\|(1+\partial_{\alpha}\hat h_n)^{-1}\|_{\beta,R_n}\le 2N^{-}$, $\ c(\hat h_n):=\left|\big<\frac1{(1+\partial_{\alpha}\hat h_n)\cdot(1+\partial_{\alpha}\hat h_n)\circ T_{-\omega\alpha}}\big>\right|\ge\frac12c$, $\displaystyle U(\hat h_n):=\bigg( \int_{\mathbb T^d}\partial_{\alpha}^{(0)}\partial_{\alpha}^{(1)}\widehat H(\sigma-\omega\alpha+\hat h_n(\sigma-\omega\alpha)\alpha,\sigma+\hat h_n(\sigma)\alpha)(1+\partial_{\alpha}\hat h_n(\sigma))(1+\partial_{\alpha}\hat h_n)(\sigma-\omega\alpha) d\sigma \bigg)^{-1}\le2U$. Therefore, $C_1$, $C_2$ in (\ref{estDelta}),(\ref{estsol}) are uniformly bounded by $C_1'$, $C_2'$.

By (\ref{estsol}), we have
\begin{align*}
     \epsilon_n&\le C_2'\nu^{-4}(\frac{R_0}{2^2}2^{-n})^{-C_3}\epsilon_{n-1}^2\\
    &\le C_2'\nu^{-4}(\frac{R_0}{2^2}2^{-n})^{-C_3}(C_2'\nu^{-4}(\frac{R_0}{2^2}2^{-n+1})^{-C_3}\epsilon_{n-2}^2)^2\\
    &\le(C_2'\nu^{-4}(\frac{R_0}{2^2})^{-C_3})^{1+2+\cdots+2^{n-1}}(2^{C_3n+C_3(n-1)\cdot2+\cdots+C_32^{n-1}})\epsilon_0^{2^n}\\
    &\le(C_2'\nu^{-4}(\frac{R_0}{2^2})^{-C_3})^{2^n}(2^{C_32^{n+1}})\epsilon_0^{2^n}\\
    &=(C_2'\nu^{-4}(\frac{R_0}{2^2})^{-C_3}2^{2C_3}\epsilon_0)^{2^n}=:(A\epsilon_0)^{2^n},
\end{align*}
where $A=C_2'\nu^{-4}(\frac{R_0}{2^2})^{-C_3}2^{2C_3}$.
Hence, if $A\epsilon_0<1$, $\epsilon_n$ decreases faster than any exponential. 

Now we show that $\mathscr E[\hat h_n]$ is well-defined during the iteration.
By (\ref{estDelta}), it follows that
\begin{align*}
    \sum_{j=0}^{n-1}\|\widehat\Delta_j\|_{\beta,R_{n}}&=\sum_{j=0}^{n-1}\|\widehat\Delta_j\|_{\beta,R_{j+1}}\le\sum_{j=0}^{n-1}C'_1\nu^{-2}(\frac{R_0}{2^2}2^{-(j+1)})^{-\frac{C_3}2}\epsilon_j\\
    &\le C'_1\nu^{-2}(\frac{R_0}{2^3})^{-\frac{C_3}2}\sum_{j=0}^{n-1}2^{\frac{C_3}{2}j}(A\epsilon_0)^{2^j}\\
    &\le C'_1\nu^{-2}(\frac{R_0}{2^3})^{-\frac{C_3}2}\sum_{j=0}^{n-1}(2^{\frac{C_3}{2}}A\epsilon_0)^{2^j}.
\end{align*}
Let $2^{\frac{C_3}{2}}A\epsilon_0\le\frac12$, then
\begin{align*}
    \sum_{j=0}^{n-1}\|\widehat\Delta_j\|_{\beta,R_{n}}\le 2C'_1\nu^{-2}(\frac{R_0}{2^3})^{-\frac{C_3}2}2^{\frac{C_3}{2}}A\epsilon_0.
\end{align*}
Hence, if $2C'_1\nu^{-2}(\frac{R_0}{2^3})^{-\frac{C_3}2}2^{\frac{C_3}{2}}A\epsilon_0\le\frac{\iota}{4}$, we have $\sum_{j=0}^{n-1}\|\widehat\Delta_j\|_{\beta,R_n}\le\frac{\iota}{4}$, $\forall\,n\in\mathbb N$. Furthermore, we have $\sum_{j=0}^{\infty}\|\widehat\Delta_j\|_{\beta,R_n}\le\frac{\iota}{4}$. By Lemma \ref{composition}, $\mathscr E[\hat h_n]$ is well-defined.

We will show that the condition numbers are uniformly bounded and we only take $N^+$ as an example. Due to \eqref{estDeltaderivative}, we obtain
\begin{align*}
    N^+(\hat h_n,R_n)&\le N^+(\hat h_{n-1},R_{n-1})+\|\partial_{\alpha}(\hat h_n-\hat h_{n-1})\|_{\beta,R_n}\\
    &\le N^+(\hat h_{n-1},R_{n-1})+4^{\beta}C'_1\nu^{-2}(\frac{R_0}{2^2}2^{-n})^{-\frac{C_3}2-\beta}\epsilon_{n-1}\\
    &\le N^+(\hat h_{n-1},R_{n-1})+4^{\beta}C'_1\nu^{-2}(\frac{R_0}{2^2})^{-\frac{C_3}2-\beta}2^{(\frac{C_3}2+\beta)n}(A\epsilon_0)^{2^{n-1}}\\
    &\le N^+(\hat h_0,R_0)+4^{\beta}C'_1\nu^{-2}(\frac{R_0}{2^2})^{-\frac{C_3}2-\beta}\sum_{j=1}^n(2^{\frac{C_3}{2}+\beta}A\epsilon_0)^{2^{j-1}}.
\end{align*}
Let $2^{(\frac{C_3}{2}+\beta)}A\epsilon_0\le\frac{1}{2}$, then
\begin{align*}
    N^+(\hat h_n,R_n)&\le N^+(\hat h_0,R_0)+4^{\beta}C'_1\nu^{-2}(\frac{R_0}{2^2})^{-\frac{C_3}2-\beta}2^{\frac{C_3}{2}+\beta+1}A\epsilon_0.
\end{align*}
Hence, if $4^{\beta}C'_1\nu^{-2}(\frac{R_0}{2^2})^{-\frac{C_3}2-\beta}2^{\frac{C_3}{2}+\beta+1}A\epsilon_0\le N^+(\hat h_0,R_0)$, then $N^+(\hat h_n,R_n)\le2 N^+(\hat h_0,R_0)$ holds for any $ n\in\mathbb N$. Therefore, condition numbers are uniformly bounded, and we have finished the inductive proof. 

\subsection{Existence of the solution}
Define $\hat h^*=\hat h_0+\sum_{j=0}^\infty\widehat\Delta_j$, then $\hat h^*$ satisfies the conditions in Theorem \ref{main thm}.

We begin by proving $\hat h^*\in G_{\beta,\frac{R_0}{2}}(\mathbb T^d)$.
By definition, we have 
\begin{align*}
    \|\hat h^*\|_{\beta,\frac{R_0}{2}}&\le\|\hat h_0\|_{\beta,\frac{R_0}{2}}+\sum_{j=0}^{\infty}\|\widehat\Delta_j\|_{\beta,\frac{R_0}{2}}\le\|\hat h_0\|_{\beta,R_0}+\sum_{j=0}^\infty \|\widehat\Delta_j\|_{\beta,R_{j+1}}<\infty.
\end{align*}
Thus, $\hat h^*\in G_{\beta,\frac{R_0}{2}}(\mathbb T^d)$.

Since $\sum_{j=0}^{\infty}\|\widehat\Delta_j\|_{\beta,R_{j+1}}\le\frac{\iota}{4}$, it follows that $\mathscr E[\hat h^*]$ is well-defined. Now, we prove that the function $\hat h^*$ is the solution of $\mathscr E[\hat h^*]=0$.
For all $n\in\mathbb N$, we have
\begin{align*}
    \|\mathscr E[\hat h^*]\|_{\beta,\frac{R_0}{2}}&\le\|\mathscr E[\hat h_n]\|_{\beta,\frac{R_0}{2}}+\|\mathscr E[\hat h^*]-\mathscr E[\hat h_n]\|_{\beta,\frac{R_0}{2}}\\
    &\le\epsilon_n+\sum_{L=0}^{\infty}\sum_{k=0}^L(L+1)M_L\|\hat h^*-\hat h_n\|_{\beta,\frac{R_0}{2}}\\
    &\le(A\epsilon_0)^{2^n}+(\sum_{L=0}^{\infty}\sum_{k=0}^L(L+1)M_L)\sum_{j=n}^\infty\|\widehat\Delta_j\|_{\beta,\frac{R_0}{2}}.
\end{align*}
Taking the limit as $n \to \infty$, we conclude that $\mathscr E[\hat h^*] = 0$.

Finally, we establish the solution estimates:
\begin{align*}
    \|\hat h^*-\hat h_0\|_{\beta,\frac{R_0}{2}}&\le\sum_{j=0}^\infty\|\widehat\Delta_j\|_{\beta,\frac{R_0}{2}}\le\sum_{j=0}^\infty\|\widehat\Delta_{j}\|_{\beta,R_{j+1}}\\
    &\le 2C'_1\nu^{-2}(\frac{R_0}{2^3})^{-\frac{C_3}2}2^{\frac{C_3}{2}}A\epsilon_0 .
\end{align*}
Therefore, $\|\hat h^*-\hat h\|_{\beta,\frac{R_0}{2}}\le C_4(N^+, N^-, c, \tau, \beta,d, \iota, R_0, T, \delta, U)\nu^{-6}\epsilon_0$.

\subsection{Local uniqueness of the solution}
Suppose that $\|\hat h^*-\hat h\|_{\beta,\frac{R_0}{2}},\,\|\hat h^{**}-\hat h\|_{\beta,\frac{R_0}{2}}\le r<\frac{\iota}{4}$, $\mathscr{E}[\hat h^*]=\mathscr{E}[\hat h^{**}]=0$, and $\langle\hat h^*\rangle=\langle\hat h^{**}\rangle=0$. Then we have $\|\hat h^{**}-\hat h^*\|_{\frac{R_0}{4}}\le\|\hat h^{**}-\hat h^*\|_{\frac{R_0}{2}}\le 2r$, and
\begin{equation}\label{uu1}
    \begin{aligned}
        0=\mathcal E[\hat h^{**}]-\mathcal E[\hat h^{*}]=&D\mathcal{E}[\hat h^*](\hat h^{**}-\hat h^*)+\sum_{L=0}^{\infty}\sum_{k=0}^L(\partial_{\alpha}^{(k)}\widehat H_L(\gamma_L^{**(-k)}(\theta))-\partial_{\alpha}^{(k)}\widehat H_L(\gamma_L^{*(-k)}(\theta))\\
        &\quad-\sum_{j=0}^L\partial_{\alpha}^{(k)}\widehat H_L(\gamma_L^{*(-k)}(\theta))(\hat h^{**}-\hat h^*)(\sigma-k\omega\alpha+j\omega\alpha))\\
        =&D\mathcal{E}[\hat h^*](\hat h^{**}-\hat h^*)+R.
    \end{aligned}
\end{equation}
By Lemmata \ref{Interpolation ineq} and \ref{composition}, we have
\begin{align*}
    \|R\|_{\beta,\frac{R_0}{4}}&\le\sum_{L=0}^{\infty}\sum_{k=0}^L M_L(L+1)^2\|\hat h^{**}-\hat h^*\|_{\beta,\frac{R_0}{4}}^2=\sum_{L=0}^{\infty}M_L(L+1)^3\|\hat h^{**}-\hat h^*\|_{\beta,\frac{R_0}{4}}^2\\
    &\le\sum_{L=0}^{\infty}M_L(L+1)^3\|\hat h^{**}-\hat h^*\|_{\beta,\frac{R_0}{8}}\|\hat h^{**}-\hat h^*\|_{\beta,\frac{R_0}{2}}\\
    &\le\sum_{L=0}^{\infty}M_L(L+1)^3\|\hat h^{**}-\hat h^*\|_{\beta,\frac{R_0}{8}}2r.
\end{align*}

Denote $\hat l^*=1+\partial_{\alpha}\hat h^*$. Since $$D\mathcal{E}[\hat h^*]\cdot \hat l^*=\frac{d}{d\theta}\mathcal{E}[\hat h^*](\theta)=0,$$
we can write the equation (\ref{uu1}) as:
\begin{equation}\label{uuu1}
    \hat l^* D\mathcal{E}[\hat h^*](\hat h^{**}-\hat h^*)-(\hat h^{**}-\hat h^*)D\mathcal E[\hat h^*]\cdot\hat l^*=-\hat l^*R.
\end{equation}

Notice that the equation (\ref{uuu1}) and the equation (\ref{modeq}) share the same form. Using the uniqueness statements for the solution of equation (\ref{modeq}) and the estimate (\ref{estDelta}), we conclude that 
\begin{equation*}
    \|\hat h^{**}-\hat h^*\|_{\beta,\frac{R_0}{8}}\le C'_1\nu^{-2}(\frac{R_0}{8})^{-\frac{C_3}{2}}\|R\|_{\beta,\frac{R_0}{4}}\le C'_1\nu^{-2}(\frac{R_0}{8})^{-\frac{C_3}{2}}\sum_{L=0}^{\infty}M_L(L+1)^3\|\hat h^{**}-\hat h^*\|_{\beta,\frac{R_0}{8}}2r.
\end{equation*}
Therefore, when $r>0$ is small enough, we obtain $\hat h^{**}=\hat h^*$. This completes the proof of local uniqueness of the solution in Theorem \ref{main thm}.

\section*{Acknowledgments}
X. Su is supported by the National Natural Science Foundation of China (Grant No. 12371186). 
The authors thank  R. de la Llave and D. Sauzin for a careful reading of the manuscripts and many helpful suggestions and comments.

\bibliographystyle{alpha}
\bibliography{reference}
\end{document}